\documentclass[12pt]{article}
\usepackage{graphicx} 
\usepackage{epstopdf}
\usepackage{subfigure}
\usepackage{color}
\usepackage{ulem}
\usepackage[english]{babel}
\usepackage{amsmath,amsthm}
\usepackage{amsfonts}
\usepackage{booktabs}
\usepackage{dsfont}

\usepackage{mathrsfs}
\usepackage{graphicx}
\usepackage{subfigure}
\usepackage{tabularx}
\newtheorem{thm}{Theorem}[section]

\newtheorem{lem}[thm]{Lemma}

\newtheorem{conj}[thm]{Conjecture}

\title{ Bipartite cubic planar graphs are dispersable}

\author {  Zeling Shao, Yanqing Liu, Zhiguo Li{$^*$}\\
{\small School of Science, Hebei University of Technology, Tianjin 300401, China}
\date{}
\footnote{Corresponding author. E-mail: zhiguolee@hebut.edu.cn}
\footnote{This work was supported in part by the Key Projects of Natural Science Research in Colleges and universities of Hebei Province,China (No.ZD2020130) and the Natural Science Foundation of Hebei Province,China (No. A2021202013). }
}

\begin{document}
\baselineskip 0.65cm

\maketitle

\begin{abstract}
 The book embedding of a graph $G$ is to place the vertices of $G$ on the spine and draw the edges to the pages so that the edges in the same page do not cross with each other.
 A book embedding is matching if the vertices in the same page have maximum degree at most 1. The matching book thickness is the minimum number of pages in which a graph can be matching book embedded. A graph $G$ is dispersable if and only if $mbt(G)=\Delta(G)$.
In this paper, we prove that  bipartite cubic planar graphs are dispersable.
\bigskip


\noindent\textbf{2000 MR Subject Classification.} 05C10
\end{abstract}

\section{Introduction}
The $book ~ embedding$ of a graph $G$ plays an important role in several areas of computer science, including fault-tolerant processor arrays, multilayer printed circuit boards (PCB), sorting with parallel stacks and Turing-machine graphs$^{[1]}$. The book embedding problem of a graph $G$ is to place all vertices of $G$ on the line, called $spine$ of a book, and assign all edges of $G$ to different half planes, called $pages$ of a book, the boundary of these half planes is spine, so that no two edges on the same page cross$^{[2]}$. The $book~ thickness$ of $G$, denoted by $bt(G)$, is the minimum number of pages in which $G$ can be book embedded. The problem of determining the book thickness of graphs has been widely studied$^{[3-6]}$.

%

A book embedding is $matching$
if maximum degree of vertices in $G$ is at most 1 on the same page$^{[7,8]}$. The $matching ~book~ thickness$ of $G$, denoted by $mbt(G)$, is defined analogously to the book thickness as the minimum number of pages in which $G$ can be matching book embedded$^{[7,8]}$. A graph $G$ with maximum degree $\Delta(G)$ is $dispersable$ if it has a proper $\Delta(G)$-edge-colouring and a $\Delta(G)$-pages book embedding so that all edges of one colour lie on the same page$^{[2,7,8]}$. By the definition of dispersable, a graph $G$ is dispersable if and only if $mbt(G)=\Delta(G)$. Complete bipartite graphs, even cycles, cubes and trees are both dispersable$^{[2,8]}$. %

In 1979, F. Bernhart and P. C. Kainen$^{[2]}$ conjectured that every $k$-regular bipartite graph is dispersable. The conjecture holds for $k\leq 2$. However,  J. M. Alam, M. A. Bekos, M. Gronemann, et al.$^{[7]}$ disproved the conjecture for the cases $k=3$ and $k=4$. Meanwhile, they proved that 3-connected bipartite cubic planar graphs are dispersable %
and proposed the following conjecture:

\begin{conj}
 All (i.e., not necessarily 3-connected) bipartite cubic planar graphs are dispersable.
\end{conj}

In this paper, we mainly consider the matching book embedding of connected bipartite cubic planar graphs and  prove the above conjecture is true. For convenience, we refer to a bipartite cubic planar graph as BCP graph for short.

\section{Preliminary}

The graphs in this paper are simple, undirected and connected. For a graph $G$, the $vertex$-$connectivity$ of $G$, denoted by $\kappa(G)$, is the minimum number of vertices whose deletion results in a disconnected graph. If $\kappa(G)\geq k$, then $G$ is $k$-$connected$.
 Similarly, the $edge$-$connectivity$ of $G$, denoted by $\kappa'(G)$, is equal to the minimum number of edges whose removal results in a disconnected graph. If $\kappa'(G)\geq k$, then $G$ is $k$-$edge$-$connected$. %

\begin{lem}$^{[9]}$
If $G$ is a graph with  minimum degree $\delta$, then $\kappa(G)\leq \kappa'(G)\leq \delta$.
\end{lem}
\begin{lem}$^{[10]}$
If $G$ is a connected bipartite cubic graph, then (1) $\kappa(G)=\kappa'(G)$; (2) $G$ is 2-connected; (3) the number of vertices in $G$ is even.
\end{lem}

\begin{figure}[htbp]
\centering
\includegraphics[height=3.6cm, width=0.3\textwidth]{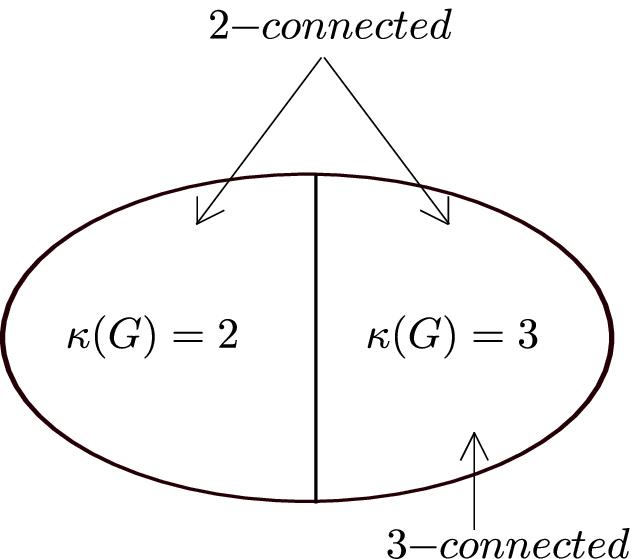}

\center{Fig.1~ Connectivity of a BCP graph $G$.}
\end{figure}

Since $\kappa(G)=\kappa'(G)$, for a BCP graph $G$, we refer to edge-connectivity as connectivity. By Lemma 2.1 and Lemma 2.2, it is easy to prove that $\kappa(G)=3$ or $\kappa(G)=2$ %
(see Fig.1).

For a 3-connected BCP graph $G$, the $sub$-$hamiltonian$ $cycle$ of $G$ is a cyclic ordering of the vertices such that when adding any missing edges between consecutive vertices the resulting graph remains planar$^{[7]}$. M. Alam, M. A. Bekos, M. Gronemann, et al. showed the following result by constructing sub-hamiltonian cycle:

\begin{lem}$^{[7]}$
The matching book thickness of a 3-connected  BCP graph is three.
\end{lem}
In order to prove the above conjecture is true, it remains  to show  the matching book thickness of a BCP graph $G$ with $\kappa(G)=2$  is three by Lemma 2.3.
\begin{figure}[htbp]
\centering
\includegraphics[height=3.6cm, width=0.74\textwidth]{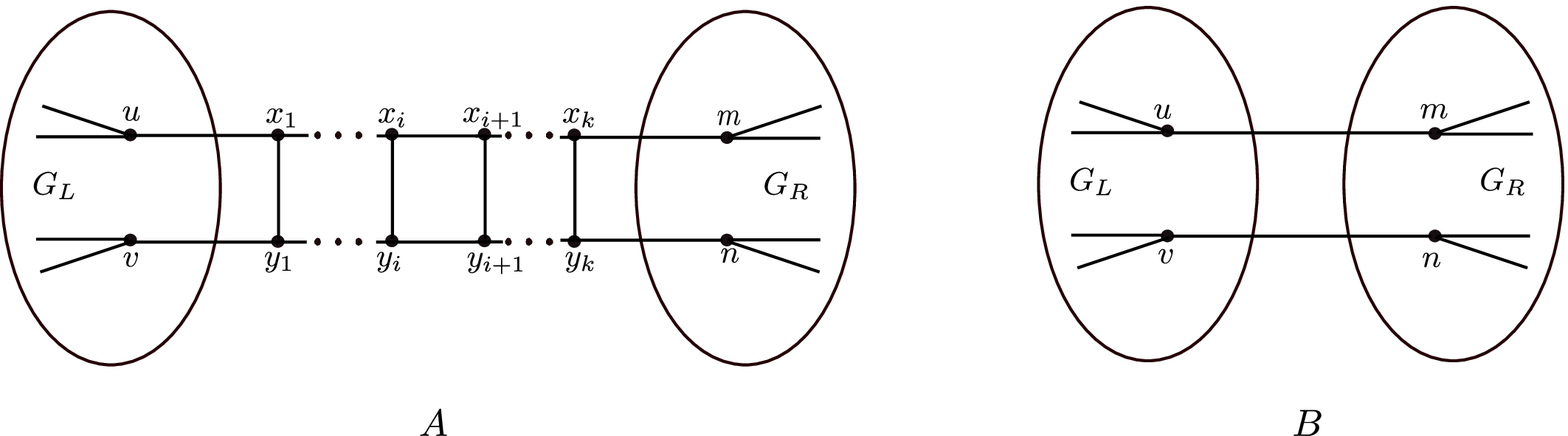}

\center{Fig.2~ (A)~The graph $G$ contains a ladder $T_k$ ($k\geq 1$). (B)~The graph $G$ contains a ladder $T_0$.}
\end{figure}

For a BCP graph $G$ with $\kappa(G)=2$, let us first introduce some definitions and notations before discussing their structure.
The $ladder$ of length $k$, denoted by $T_k$, is a graph obtained from two copies of the path with respective vertex sequences $x_1,\cdots, x_k$ and $y_1,\cdots, y_k$ by joining $x_i$ to $y_i$ for every $i\in \{1,2,3,\cdots,k\}$ (see [10]). Suppose that $G_L$ and $G_R$ are two disjoint graphs, and $u,v\in V(G_L)$, $m,n\in V(G_R)$. We say that the graphs $G_L$ and $G_R$ are joined by a ladder $T_k$ at vertices $u,v, m$ and $n$, and form a new graph $G$ if there exists either a ladder $T_k$ of $G$ for $k\geq 1$ such that $\{(u,x_1),(v,y_1),(m,x_k),(n,y_k)\}\subset E(G)$, or a ladder $T_k$ of $G$ for $k=0$ such that $\{(u,m),(v,n)\}\subset E(G)$, see Fig.2. For convenience, we denote  $G$ by $M(G_L, T_k, G_R)$, $k\geq 0$. %

\begin{lem}$^{[11]}$
If $G$ be a cubic graph with $\kappa'(G)=2$, then there exist two disjoint subgraphs $G_L$ and $G_R$ of G and a ladder $T_k$ for some $k\in \mathbb{N}$ such that $G=M(G_L, T_k, G_R)$.

\end{lem}

Given a cubic graph $G$ with $\kappa(G)=2$, by Lemma 2.2, $\kappa'(G)=\kappa(G)$. By Lemma 2.4, there exist two disjoint subgraphs $G_{L}$ and $G_{R}$  of $G$ and there are non-adjacent pairs of vertices $u,v\in V(G_{L})$ and $m,n\in V(G_{R})$ such that $G$ consists of  the disjoint graphs $G_{L}$ and $G_{R}$ joined by a ladder $T_k$ $(k\geq 0)$ at the vertices $u,~v,~m$ and $n$.
By adding a edge $(u,v)$ to $G_L$ and adding a edge $(m,n)$ to $G_R$, we can get two new graphs $G_{L}':=G_{L}+ (u, v)$ and $G_{R}':=G_{R}+ (m, n)$. %

\begin{lem}$^{[10]}$  Let $G$ be a connected BCP graph with $\kappa(G)=2$. For any subgraphs $G_{L}$, $G_{R}$ and $T_k$ of $G$ such that $G=M(G_L, T_k, G_R)$, where $k\geq 0$, the graphs $G_{L}'$ and $G_{R}'$ are both 2-connected BCP graphs.
\end{lem}

\section{Placement of vertices on the spine for the BCP graphs}

For a BCP graph $G$,  we consider the placement of vertices on the spine from two situations.

\noindent\textbf{Case 1: The BCP graph $G$ with order at most 24.}

The smallest non-hamiltonian BCP graph has 26 vertices$^{[12]}$. Hence every BCP graph with vertices at most 24 is a hamiltonian graph. We put all vertices of $G$ along the spine with a hamiltonian ordering.

\noindent\textbf{Case 2: The BCP graph  $G$ with order at least 26 and connectivity two.}

Firstly, according to Lemma 2.4 and Lemma 2.5,  we can get a ternary tree decomposition of $G$ with each leaf is a BCP graph with order at most 24, a 3-connected BCP graph with order at least 26 or a ladder by applying following two steps for the graph $G$ by iteration (see Fig.3):

\noindent
{\textbf{Step~1.}}~If $|V(G)|\geq 26$ and $\kappa(G)=2$, by Lemma 2.4, we decompose $G$ into $G_L,T_k$ and $G_R$, otherwise stop.

\noindent
{\textbf{Step~2.}}~Let $G_{L}'=G_{L}+ (u, v)$ and $G_{R}'=G_{R}+ (m, n)$ and go to step 1 for $G=G_{L}'$ and $G=G_{R}'$ respectively.

%

\begin{figure}[htbp]
\centering
\includegraphics[height=4.5cm, width=0.36\textwidth]{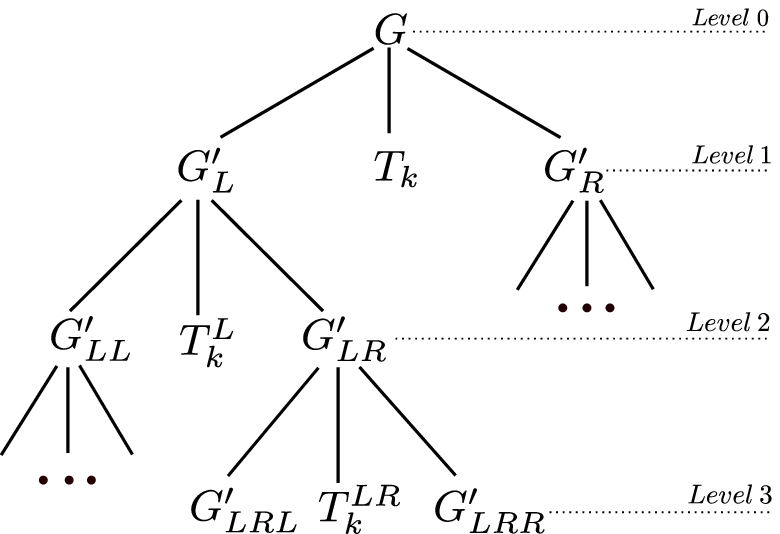}

Fig.3~ A ternary tree decomposition of a BCP graph $G$ with $|V(G)|\geq 26$ and $\kappa(G)=2$.
\end{figure}


Secondly, we use vertex orderings of leaves in a ternary tree decomposition of $G$ to get a vertex ordering of $G$ on the spine. The $level$ of a node in a ternary tree decomposition of $G$ is the length from 
the root graph $G$ to the node. 
The $height$ of a ternary tree decomposition of $G$ is defined as the maximum depth of any leaf node from the root node, denoted by $h(G)$.  Suppose $h(G)=t$, let $G'_{\beta_j}=M(G_{{\beta_j}L}, T^{\beta_j}_k, G_{{\beta_j}R})$ be a node at the ${j}$th level in the tree-decomposition of $G$ where $\beta_j$ is a length-${j}$ sequence of $L$s and $R$s, $1\leq j\leq t-1$. In particular, if $j=t$, then $\beta_t=\beta_{(t-1)}L$ or $\beta_t=\beta_{(t-1)}R$.


\begin{lem}$^{[13]}$
Let $G$ be a simple graph.

(1) If a graph $G$ has an n-book matching embedding with printing cycle $v_1$, $\cdots$, $v_p$, then $G$ also has an n-book matching embedding with printing cycle $v_2$, $\cdots$, $v_p$, $v_1$.

(2) If a graph $G$ has an n-book matching embedding $\alpha$ with printing cycle $v_1$, $\cdots$, $v_p$, then $G$ also has an n-book matching embedding $\alpha^-$ with printing cycle $v_p$, $\cdots$, $v_1$.
\end{lem}

Given a BCP graph $G$ with order at least 26 and  connectivity
two, and  suppose we have obtained a ternary tree decomposition of $G$, let us consider  the procedure of  putting the vertices of $G$ on the spine.

\noindent
{\textbf{Step~1.}}~If $V(G'_{\beta_{t}})\leq 24$, a vertex ordering of $G'_{\beta_{t}}$ on the spine is a hamiltonian ordering, otherwise, take a sub-hamiltonian ordering. If $k$ is even, a vertex ordering of a ladder $T^{\beta_{t-1}}_k$ ($k\geq 1$) on the spine is $y_1$, $x_1$, $x_2$, $y_2$, $\cdots$, $y_{k-1}$, $x_{k-1}$, $x_k$, $y_k$; if $k$ is odd, a vertex ordering of a ladder $T^{\beta_{t-1}}_k$ ($k\geq 1$) on the spine is $y_1$, $x_1$, $x_2$, $y_2$, $\cdots$, $x_{k-1}$, $y_{k-1}$, $y_k$, $x_k$.

\noindent
{\textbf{Step~2.}}~According to Lemma 3.1, we can obtain a vertex ordering of $G'_{\beta_{t-1}}$ by modifying vertex orderings of $G'_{\beta_{t}}$ such that edges $(u,x_1)$ and $(v,y_1)$ are nested, edges $(m,x_k)$ and $(n,y_k)$ are nested, a vertex $v$ is placed next to a vertex $y_1$, a vertex $n$ is placed next to a vertex $y_k$ (or a vertex $m$ is placed next to a vertex $x_k$).

In particular, for a ladder $T^{\beta_{t-1}}_0$,  edges $(u,m)$ and $(v,n)$ are nested and a vertex $u$ is placed next to a vertex $m$ (or a vertex $v$ is placed next to a vertex $n$).

\noindent
{\textbf{Step~3.}}~Applying the above two steps for $G'_{\beta_j}$ by iteration from $G'_{\beta_{t-2}}$ to $G'_{\beta_1}$ until we obtain a vertex ordering of $G$ on the spine.


\section{Main results}

 In this section, we prove the main results by colouring methods.  Let us firstly recall some definitions from graph colouring.

The chromatic number of a graph $G$, denoted by $\chi(G)$, is the minimum number of colours need to colour the vertices of $G$ in such a way that the neighboring vertices receive distinct colours. If $\chi(G)\leq k$, then $G$ is $k$-colourable. A graph $G$ is $k$-edge-colourable if the edges of $G$ receive $k$ colours such that no two adjacent edges are assigned the same colour. The minimum $k$ for which $G$ is $k$-edge-colourable is called its edge chromatic number, and denoted $\chi'(G)$. Similarly, a graph $G$ is $k$-face-colourable if the faces of $G$ receive $k$ colours such that the adjacent faces receive distinct colours. The minimum $k$ for which $G$ is $k$-face-colourable is called its face chromatic number, and denoted $\chi_f(G)$.

%

\begin{lem}$^{[14]}$
A planar triangulation is 3-colourable if and only if every vertex has even degree.
\end{lem}

\begin{lem}
If $G$ is a BCP graph, then $\chi_f(G)=3$.
\end{lem}
\begin{proof}
By the definition of the dual graph, $\chi_f(G)=\chi(G^{*})$ and  each vertex of $G^*$ has even degree. It is well known that a simple connected planar graph is a triangulation if and only if its dual graph  is cubic. Hence $G^*$ is a planar triangulation. By Lemma 4.1, we have $\chi(G^{*})=3$. Therefore $\chi_f(G)=3$.
\end{proof}

\begin{lem}
Let $G$ be a BCP graph. Then $\chi'(G)=3$.
\end{lem}
\begin{proof}
By Lemma 4.2, we have that $\chi_f(G)=3$.  For convenience, three colours are denoted by the symbols $E_1,~E_2$ and $E_3$.  These colours are assigned to three faces incident to one vertex. We obtain a proper 3-edge-colouring of $G$ by assigning to each edge the sum of colours of the adjacent pairs of face, i.e., $E_1+E_2,~ E_1+E_3$ and $E_2+E_3$ (see Fig.4). Any two adjacent edges receive the different colour. Since $\chi'(G)\geq \Delta (G)=3,$ we have $\chi'(G)=3$.
\end{proof}

\begin{figure}[htbp]
\centering
\includegraphics[height=3.5cm, width=0.24\textwidth]{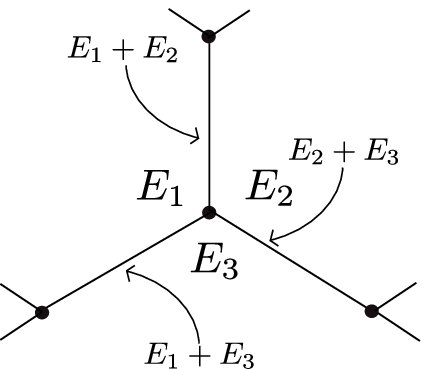}

Fig.4~ A 3-edge-colouring of a BCP graph induced by a 3-face-colouring.
\end{figure}

\begin{lem}
Let $T_k$ ($k\geq 1$) be a ladder. Then $mbt(T_k)=3$.
\end{lem}
\begin{proof}
Vertex ordering and edge embedding of $T_k$ are showed in Fig.5. By the definition of matching book embedding of $T_k$ ($k\geq 1$), we have that $mbt(T_k)\geq 3$.
Hence $mbt(T_k)=3$, $k\geq 1$.
\end{proof}

\begin{figure}[htbp]
\centering
\includegraphics[height=2.4cm, width=0.66\textwidth]{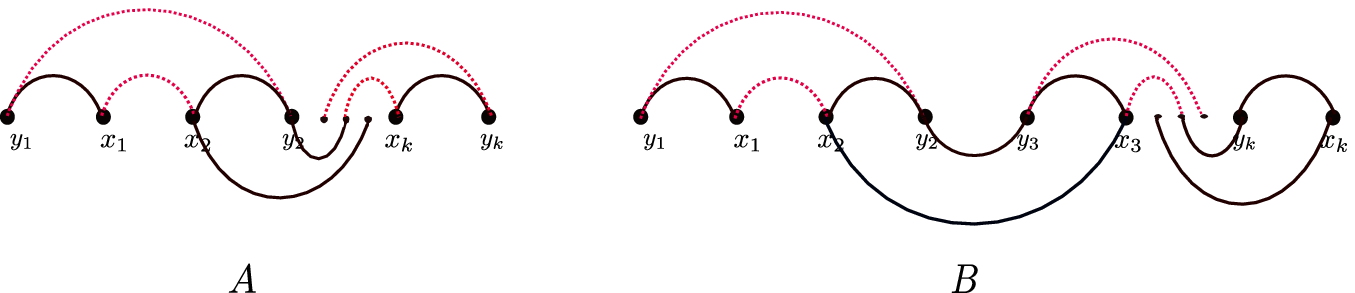}

Fig.5~Matching book embeddings of a ladder $T_k$ ($k\geq 1$). (A) $k$ is even. (B) $k$ is odd.
\end{figure}

\begin{figure}[htbp]
\centering
\includegraphics[height=6. cm, width=0.5\textwidth]{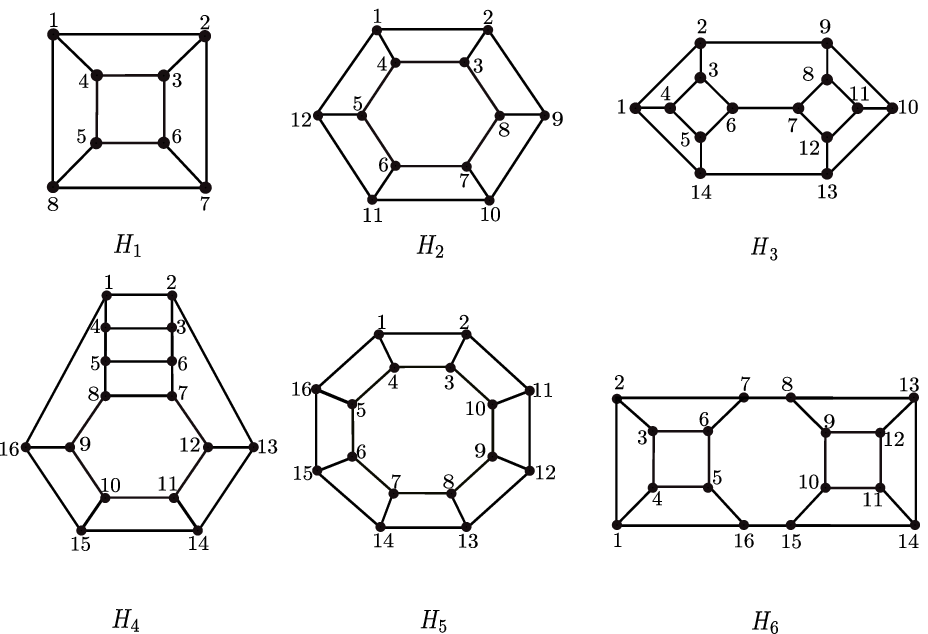}

Fig.6~ All BCP graphs with order at most 16.
\end{figure}
Now we investigate the matching book thickness of a BCP graph with connectivity two from two cases.

\textbf{Case \uppercase\expandafter{\romannumeral1}: The BCP graph $G$ with order at most 24 and connectivity two.}

Let $G=M(G_L,T_k, G_R)$ ($k\geq 0$) be a BCP graph with $|V(G)|\leq 24$ and $\kappa(G)=2$. Since the smallest BCP graph is $H_1$ (see Fig.6), by Lemma 2.4 and Lemma 2.5, there exist BCP graphs $G_L'$ and $G_{R}'$ with order at most 16. For more details about the BCP graphs with order at most 16, we refer to [15]. Fig.6 indicates all BCP graphs with order at most 16, and Fig.7 shows matching book embeddings of these graphs. %

\begin{figure}[htbp]
\centering
\includegraphics[height=10.9cm, width=0.9\textwidth]{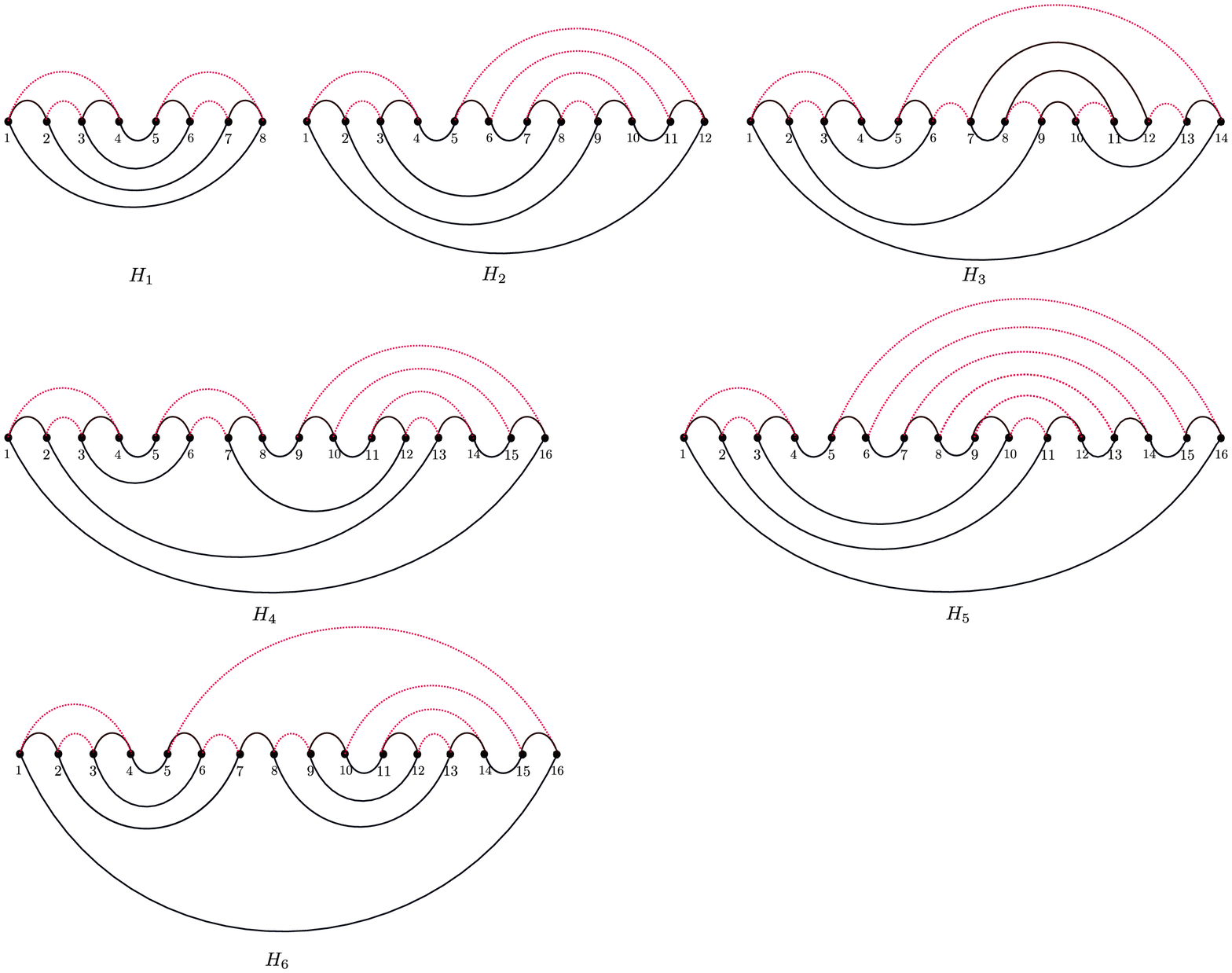}

Fig.7~ Matching book embeddings of the graphs in Fig.6.
\end{figure}

\begin{lem}
Let $G$ be a BCP graph with $|V(G)|\leq 24$ and $\kappa(G)=2$. Then $mbt(G)=3$.
\end{lem}

\begin{proof}
If $G=M(G_L,T_k,G_R)$ is a BCP graph with $|V(G)|\leq 24$ and $\kappa(G)=2$, then the vertices of $G$ are placed on the spine with a hamiltonian ordering.

Firstly, if $V(T_k)=\{x_1,x_2,\cdots,x_k,y_1,y_2,\cdots,y_k\}$,   $k\geq 1$,  without loss of generality,
we suppose that $\{(u,x_1),(v,y_1),(m,x_k),(n,y_k)\}\subset E(G)$.
By Lemma 4.2 and Lemma 4.3, if $k$ is even, then the edges $(u,x_1)$, $(v,y_1)$, $(m,x_k)$ and $(n,y_k)$ receive the same colour $E_1+E_2$.
Since $G'_L$ and $G'_R$ are both BCP graphs (see Fig.6), we colour the edges $(u,v)$ and $(m,n)$ with the colour $E_1+E_2$. On the other hand, if $k$ is odd, then the edges $(u,x_1)$ and $(v,y_1)$ receive the same colour $E_1+E_2$, and the edges $(m,x_k)$ and $(n,y_k)$ have the same colour $E_1+E_3$. %
Similarly, we colour the edge $(u,v)$ with $E_1+E_2$ and colour the edge $(m,n)$ with $E_1+E_3$.

By Lemma 2.4 and Lemma 2.5, there exist graphs $G'_L$, $G'_R$ and $T_k$, $k\geq 1$. The matching book embeddings of $G'_L$ and $G'_R$ are showed in Fig.7, and $mbt(G'_L)\leq 3$ and $mbt(G'_R)\leq 3$. According to Lemma 4.4, we have that $mbt(T_k)=3$. After removing edges ($u,v$) and ($m,n$) from $G$, the edges $(u,x_1)$, $(v,y_1)$, $(m,x_k)$ and $(n,y_k)$ are drawn on  pages. If $k$ is even, we assign the edges $(u,x_1)$, $(v,y_1)$, $(m,x_k)$ and $(n,y_k)$ to the page corresponding to the colour $E_1+E_2$. If $k$ is odd, the edges $(u,x_1)$ and $(v,y_1)$ are drawn on the page corresponding to the colour $E_1+E_2$ and the edges $(m,x_k)$ and $(n,y_k)$ are drawn on the page corresponding to the colour $E_1+E_3$. Since the vertex $u$ is next to $x_1$, $v$ is next to $y_1$, $m$ is next to $x_k$ and $n$ is next to $y_k$ on the spine, the edges that receive the same colour are matching book embedded in the same page without crossing.

One can prove the matching book embedding of a BCP graph with $T_0$ in the same way.

By the definition of matching book embedding of a graph $G$, $mbt(G)\geq \Delta(G)$.  %
Therefore, the BCP graphs with $|V(G)|\leq 24$ and $\kappa(G)=2$ are dispersable.
\end{proof}

\begin{figure}[htbp]
\centering
\includegraphics[height=5 cm, width=0.6\textwidth]{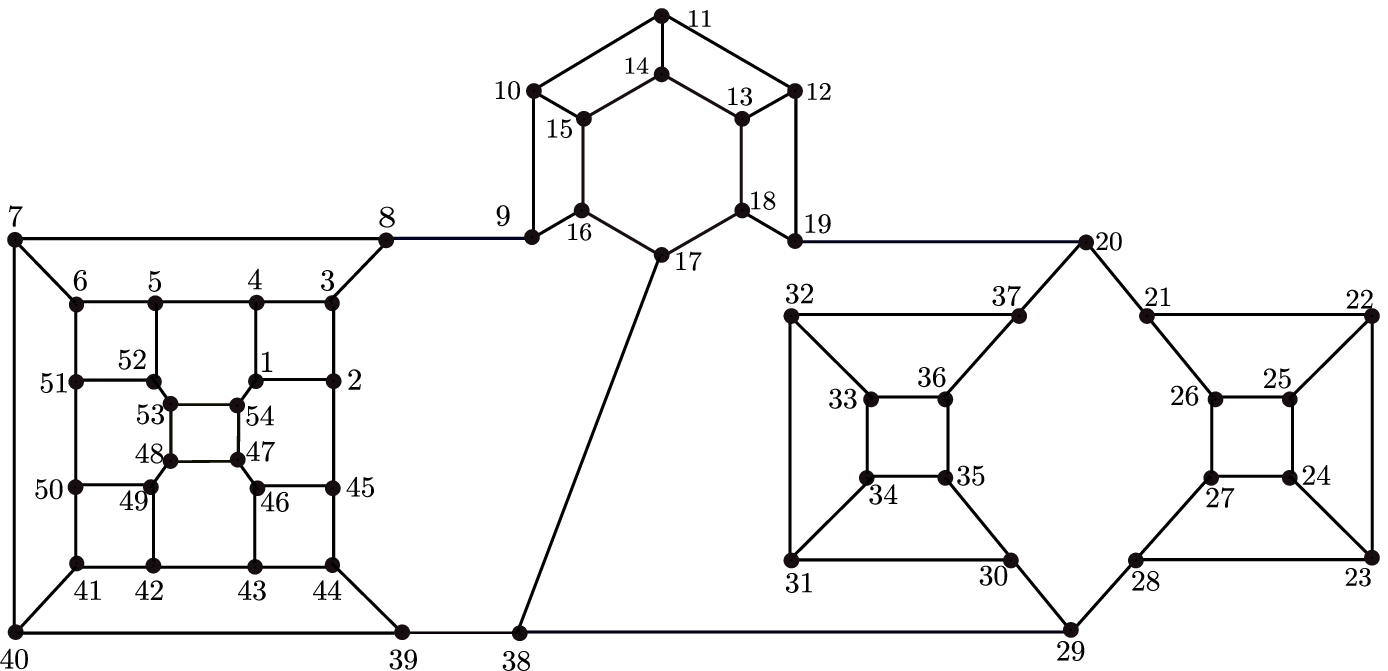}

\center{Fig.8~ A BCP graph $H$ with $|V(H)|=54$ and $\kappa(H)=2$.}
\end{figure}
\noindent\textbf{Case \uppercase\expandafter{\romannumeral2}: The BCP graph $G$ with order at least 26 and connectivity two.}

\begin{lem}
Let $G$ be a BCP graph with $|V(G)|\geq 26$ and $\kappa(G)=2$. Then $mbt(G)=3$.
\end{lem}

\begin{proof}
Let $G=M(G_L,T_k,G_R)$ be a BCP graph with $|V(G)|\geq 26$ and $\kappa(G)=2$. All vertices of $G$ are placed on the spine according to the placement of vertices in Case 2 of section 3.

Firstly, if $V(T_k)=\{x_1,x_2,\cdots,x_k,y_1,y_2,\cdots,y_k\}$,   $k\geq 1$,  without loss of generality,
we suppose that $\{(u,x_1),(v,y_1),(m,x_k),(n,y_k)\}\subset E(G)$. The colouring of edges $(u,x_1)$, $(v,y_1)$, $(m,x_k)$ and $(n,y_k)$ in $G$ is similar to the colouring of Lemma 4.5.

According to a ternary tree decomposition of $G$, some graphs are obtained in the end, and each graph is a BCP graph with order at most 24, a 3-connected BCP graph with order at least 26 or a ladder. By Lemma 2.3, Lemma 4.4 and Lemma 4.5, these graphs have matching book thickness three.
After removing edges ($u,v$) and ($m,n$) from $G$, the edges $(u,x_1)$, $(v,y_1)$, $(m,x_k)$ and $(n,y_k)$ are drawn on pages. If $k$ is even, the edges $(u,x_1)$, $(v,y_1)$, $(m,x_k)$ and $(n,y_k)$ are assigned to the page corresponding to the colour $E_1+E_2$. If $k$ is odd, we assign the edges $(u,x_1)$ and $(v,y_1)$ to the page corresponding to the colour $E_1+E_2$ and assign the edges $(m,x_k)$ and $(n,y_k)$ to the page corresponding to the colour $E_1+E_3$. %
According to the placement of vertices, the edges $(u,x_1)$ and ($v,y_1$) are nested, the edges $(m,x_k)$ and $(n,y_k)$ are nested,  no two edges of the same page cross each other. %

One can prove the matching book embedding of a BCP graph with $T_0$ in the same way.

Since $mbt(G)\geq \Delta(G)=3$, we have that $mbt(G)=3$. Thus we conclude that  BCP graphs with $|V(G)|\geq 26$ and $\kappa(G)=2$ are dispersable.

\end{proof}

Fig.8 shows a non-hamiltonian BCP graph $H$ with $|V(H)|=54$ and $\kappa(H)=2$. The matching book embedding of $H$ is illustrated in Fig.9.

By Lemma 2.3, Lemma 4.5 and Lemma 4.6, we obtain the following theorem.

\begin{thm}
All bipartite cubic planar graphs are dispersable.
\end{thm}

\begin{figure}[htbp]
\centering
\includegraphics[height=6.5cm, width=1\textwidth]{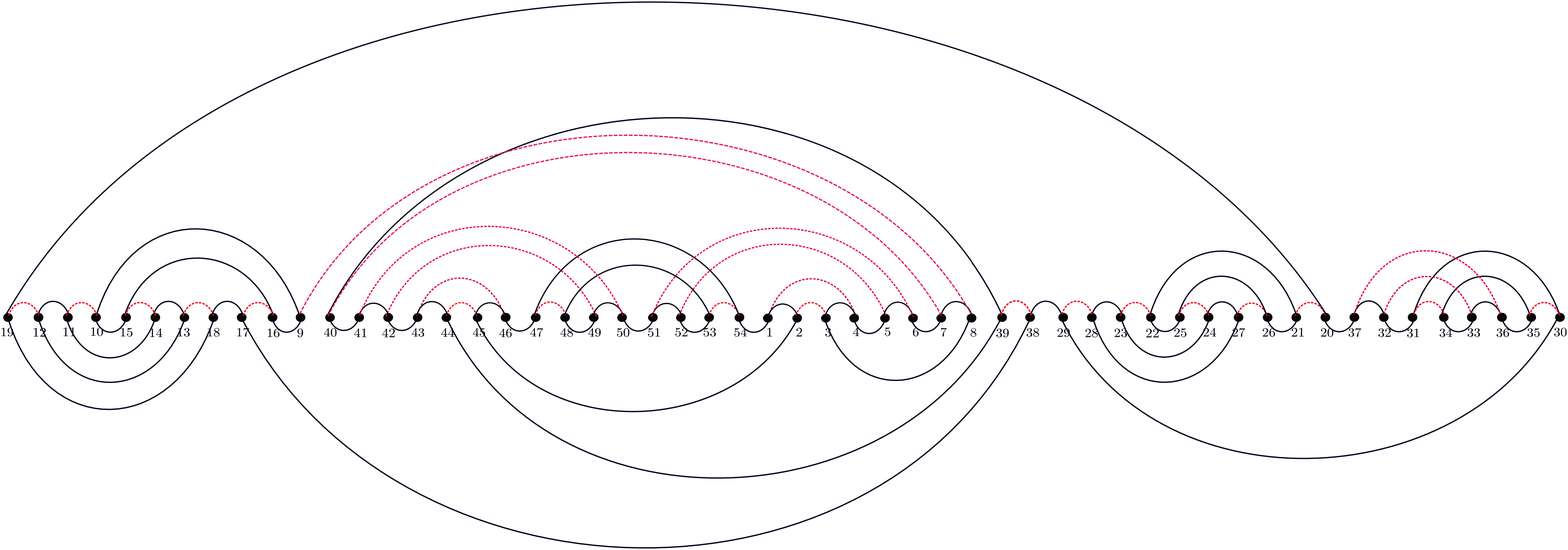}

Fig.9~ A matching book embedding of $H$ in Fig.8.
\end{figure}


\begin{thebibliography}{99}
\bibitem{AB}
F. R. K. Chung, F. T. Leighton, A. L. Rosenberg. Embedding Graphs in Books: A Layout Problem with Applications to VLSI Design. Society for Industrial and Applied Mathematics, 1987, 8(1): 33-58.

\bibitem{ZX}
F. Bernhart, P. C. Kainen. The Book Thickness of a Graph. Journal of Combinatorial Theory, Series B, 1979, 27: 320-331.


\bibitem{BS}
V. Dujmovic, D. R. Wood. Graph Treewidth and Geometric Thickness Parameters. Discrete Computational Geometry, 2007, 37(4): 641-670.



\bibitem{GH}
M. A. Bekos, M. Gronemann,  C. N. Raftopoulou. Two-Page Book Embeddings of 4-Planar Graphs. Algorithmica, 2016, 75(1): 158-185.


\bibitem{GH}
S. Istrail. An Algorithm for Embedding Planar Graphs in Six Pages. Iasi University Annals, Mathematics-Computer Science, 1988, 34(4): 329-341.


\bibitem{GH}
M. Yannakakis. Embedding Planar Graphs in Four Pages. Journal of Computer and System Sciences, 1989, 38(1): 36-67.

\bibitem{BK}
J. M. Alam, M. A. Bekos, M. Gronemann, M. Kaufmann and S. Pupyrev. On Dispersable Book Embeddings. International Workshop on Graph-Theoretic Concepts in Computer Science, 2018, 11159: 1-14.


\bibitem{AB}
S. Overbay. Generalized Book Embeddings. Ph. D. thesis, Fort Collins: Colorado State University, 1998.



\bibitem{BM}
J. A. Bondy, U. S. R. Murty. Graph Theory.  London: Springer, 2008.



\bibitem{BM}
 G. Abrishami,  F. Rahbarnia. A Note on the Smallest Connected Non-Traceable Cubic Bipartite Planar Graph. Discrete Mathematics, 2019, 342(8): 2384-2392.

\bibitem{BM}  
J. A. Bondy, U. S. R. Murty. Graph Theory with Applications. Elsevier Science Publishing, 1982.
%
\bibitem{BM}
T. Asano, N. Saito, G. Exoo and F. Harary. The Smallest 2-connected Cubic Bipartite Planar Nonhamiltonian Graph. Discrete Mathematics, 1982, 38: 1-6.

%



\bibitem{BM}
Z. L. Shao, Y. Q. Liu, Z. G. Li. The Matching Book Embedding of Biconnected Outerplanar Graphs.  arXiv: 2008.13324v1

\bibitem{BM}
P.J. Heawood. On the Four-color Map Theorem. The Quarterly Journal of Pure and Applied Mathematics, 1898, 29: 270-285.


\bibitem{BM}
 D. L. Peterson. A Note on Hamiltonian Cycles in Bipartite Plane Cubic Maps Having Connectivity 2. Discrete Mathematics, 1981, 36(3): 327-332.

\end{thebibliography}
\end{document}